\documentclass[11pt,a4paper]{article}
\usepackage[T1]{fontenc}
\usepackage[utf8]{inputenc}
\usepackage[english]{babel}
\usepackage{amsmath,amssymb,amsthm,mathtools}
\usepackage{newpxtext,newpxmath}
\usepackage{microtype}
\usepackage[a4paper,left=27mm,right=27mm,top=25mm,bottom=26mm,headsep=7mm]{geometry}
\usepackage{booktabs,array}
\usepackage{enumitem}
\usepackage{fancyhdr}
\usepackage[hidelinks,bookmarksnumbered]{hyperref}
\usepackage{bookmark}
\usepackage{needspace}
\fancypagestyle{plain}{\fancyhf{}\fancyfoot[C]{\small\thepage}}
\allowdisplaybreaks[1]
\numberwithin{equation}{section}
\newtheorem{theorem}{Theorem}[section]

\newtheorem{lemma}[theorem]{Lemma}
\newtheorem{corollary}[theorem]{Corollary}
\theoremstyle{definition}
\newtheorem{example}[theorem]{Example}
\theoremstyle{remark}
\newtheorem{remark}[theorem]{Remark}
\newcommand{\E}{\mathbb E}
\newcommand{\Pp}{\mathbb P}
\newcommand{\R}{\mathbb R}
\newcommand{\Z}{\mathbb Z}
\newcommand{\Fq}{\mathbb F_q}
\newcommand{\one}{\mathbf 1}
\newcommand{\cH}{\mathcal H}
\newcommand{\cA}{\mathcal A}
\newcommand{\cL}{\mathcal L}
\newcommand{\eps}{\varepsilon}
\newcommand{\as}{\asymp}
\newcommand{\norm}[1]{\lVert #1\rVert}
\newcommand{\abs}[1]{\lvert #1\rvert}
\DeclareMathOperator{\diam}{diam}
\DeclareMathOperator{\KL}{KL}
\DeclareMathOperator{\RD}{RD}
\DeclareMathOperator{\Law}{Law}
\newcommand{\step}[1]{\par\smallskip\noindent\textit{#1.}\ }
\hypersetup{
 pdftitle={Stationary Bernoulli processes and overlap profiles},
 pdfsubject={An invariant comparison formula, dyadic entropy, and applications to moving windows and finite geometry},
 pdfkeywords={Bernoulli processes, stationary processes, Haar measure, overlap profiles, entropy, scan windows},
 pdfcreator={LaTeX with hyperref}}

\begin{document}
\title{Dudley-type bound for stationary Bernoulli processes}

\author{Witold Bednorz\footnote{University of Warsaw, Poland. wbednorz@mimuw.edu.pl}, Rafa{\l}  Martynek \footnote{Warsaw School of Economics, Poland. rmartyn@sgh.waw.pl}  and Rafa{\l} Meller \footnote{University of Warsaw, Poland. r.meller@mimuw.edu.pl}}

\date{}
\maketitle
\vspace{3pt}
\begin{abstract}
This is preliminary, AI-generated note on the Dudley-type bound for stationary Bernoulli processes postulated by M. Talagrand \cite{talagrand1} based on the unpublished notes of the authors and the recent note \cite{invariant} which in turn is an alternative proof of the result \cite{Hu}.
\end{abstract}

\section{Introduction and the invariant formula}\label{sec:intro}

We study expected suprema of Bernoulli processes whose coefficient vectors form a finite orbit under coordinate permutations. Using the prescribed-index-law comparison \cite{invariant}, we express the expected supremum through a single invariant partition function. The probabilistic comparison is an external input; the orbit reduction is proved by symmetrizing a maximizing index and computing a Gibbs self-coupling. We then identify the equivalent dyadic ball-volume formula, including its level-zero term. For indicator vectors, the partition function depends only on the distribution of overlaps between sets. This gives a common treatment of separated set orbits, independent blocks, moving intervals, rectangular windows, and affine subspaces over finite fields. We keep signed and absolute suprema distinct, resolve the full-window endpoint, and explain how the same overlap profile produces different estimates for Bernoulli, Gaussian, and symmetric exponential coordinates.

Let $G$ be a finite group, let $f:G\to\R$, and let $(\eps_x)_{x\in G}$ be independent random signs, with
$\Pp(\eps_x=1)=\Pp(\eps_x=-1)=1/2$. Consider
\[
 X_g=\sum_{x\in G}f(gx)\eps_x,\qquad g\in G.
\]
The question is to estimate $\E\max_g X_g$ and $\E\max_g\abs{X_g}$ from deterministic information about $f$. The process is stationary under right translations of its index: for every $h\in G$, the vectors $(X_{gh})_{g\in G}$ and $(X_g)_{g\in G}$ have the same law. This follows by permuting the independent signs. The variables $X_g$ themselves are generally not independent.

When $f=\one_A$, the problem concerns random sums over translates of a fixed set. Nearby translates may share almost all their coordinates, so counting the translates alone can substantially overestimate the maximum. The useful information is how their overlaps are distributed. The purpose of this paper is to make that principle explicit.

\subsection{Background and scope}

The geometric study of expected suprema was shaped by Dudley's entropy bound for Gaussian processes \cite{dudley}. Fernique's work on stationary Gaussian processes and majorizing measures \cite{fernique} identified the special role of invariant geometry. Talagrand's majorizing-measure theorem \cite{talagrand} completed the characterization for general Gaussian processes. In this setting, the canonical distance records all increment distributions, and the problem is to organize that metric information across scales.

For Bernoulli processes, a purely quadratic description misses the boundedness of the coordinates. The general characterization proved by Bednorz and Lata\l a \cite{bernoulli} separates a part controlled by an $\ell^1$ bound from a part controlled by Gaussian chaining. The invariant comparison of Bednorz, Martynek and Meller \cite{invariant} provides a different interface: it compares a prescribed-law pairing with an integral of entropy--distortion costs. On a finite orbit, the marginal constraints in that variational problem can be solved exactly by symmetry.

The present paper develops this consequence and its applications. It is an expanded account of the supplied stationary note, not an independent proof of the general Bernoulli comparison. The non-elementary input is stated precisely in Section~\ref{sec:proof}; all subsequent orbit, dyadic, and overlap calculations are given here. The dyadic result recovers the formulation of \cite{stationary}. No novelty claim is made for the individual maximal estimates in the examples.

\subsection{The main formula}

Throughout, $I$ is a finite coordinate set and $\Gamma$ is a finite group of coordinate permutations. Let $\Omega\subset\R^I$ be one nonempty orbit of $\Gamma$, and let $(\eps_i)_{i\in I}$ be independent symmetric signs. Write
\[
 S(\Omega)=\E\max_{a\in\Omega}\langle a,\eps\rangle,
 \qquad
 S_{\mathrm{abs}}(\Omega)=\E\max_{a\in\Omega}\abs{\langle a,\eps\rangle}.
\]
All logarithms are natural. The notation $A\lesssim B$ means $A\le CB$ for a universal constant $C$, and $A\as B$ means both inequalities. A subscript records the permitted dependence of the constants.

The coordinate cost is quadratic at small arguments and bounded at large arguments:
\begin{equation}\label{eq:cost}
 \chi(u)=\min\{u^2,1\},\qquad
 D_t(a,b)=\sum_{i\in I}\chi\!\left(\frac{\abs{a_i-b_i}}{t}\right),
 \quad t>0.
\end{equation}
For any $a\in\Omega$, define
\begin{equation}\label{eq:partition}
 Z_\Omega(t)=\frac1{\abs{\Omega}}\sum_{b\in\Omega}e^{-D_t(a,b)},
 \qquad
 \cH(\Omega)=\int_0^\infty-\log Z_\Omega(t)\,dt.
\end{equation}
The value of $Z_\Omega(t)$ does not depend on $a$: a permutation carrying $a$ to another reference point also permutes the summation indices and preserves $D_t$.

\begin{theorem}[Invariant Bernoulli formula]\label{prop:orbit}
There are universal comparison constants such that
\begin{equation}\label{eq:orbit}
 S(\Omega)\as\cH(\Omega).
\end{equation}
For every $a_0\in\Omega$,
\begin{align}
 \E\max_{a\in\Omega}\abs{\langle a-a_0,\eps\rangle}
   &\as\cH(\Omega),\label{eq:anchor}\\
 S_{\mathrm{abs}}(\Omega)
   &\as\norm{a_0}_2+\cH(\Omega).\label{eq:absolute}
\end{align}
These statements include singleton orbits: then $\cH(\Omega)=S(\Omega)=0$, while the absolute supremum is comparable to $\norm{a_0}_2$.
\end{theorem}

The two terms in \eqref{eq:absolute} have different meanings. The entropy integral measures variation across the orbit. The norm term measures the size of one random variable, which need not disappear when the orbit has no variation. Confusing these roles would already give the wrong answer for a constant function $f$.

The integrals are finite. Writing $\E_{b\in\Omega}$ for uniform averaging, the diagonal summand gives $Z_\Omega(t)\ge\abs{\Omega}^{-1}$, while Jensen's inequality gives
\[
 -\log Z_\Omega(t)
 \le \frac1{\abs{\Omega}}\sum_{b\in\Omega}D_t(a,b)
 \le \frac{\E_{b\in\Omega}\norm{a-b}_2^2}{t^2}.
\]

For the group-indexed process above, take $a_g=(f(gx))_{x\in G}$. Coordinate permutations $x\mapsto hx$ send $a_g$ to $a_{gh}$, so these vectors form a single orbit. Each distinct vector has the same number of representatives in $G$, by the orbit--stabilizer identity. Consequently
\begin{equation}\label{eq:regular}
 Z_f(t)=\frac1{\abs G}\sum_{g\in G}
 \exp\!\left\{-\sum_{x\in G}
 \min\!\left(\frac{\abs{f(gx)-f(x)}^2}{t^2},1\right)\right\}.
\end{equation}
Thus $\E\max_g X_g\as\int_0^\infty-\log Z_f(t)\,dt$, and the absolute version has the additional term $\norm f_2$, with respect to unnormalized counting measure. No bound on $\norm f_\infty$ is required.

The proof has two separate parts. First, symmetry makes the uniform index law optimal. Second, the same symmetry makes the Gibbs kernel a coupling with both required marginals. Section~\ref{sec:dyadic} converts the resulting integral to dyadic ball volumes. Sections~\ref{sec:overlap}--\ref{sec:finite} apply it to overlap geometry, and Section~\ref{sec:tails} changes the coordinate law without changing the sets.

\section{Why the variational problem simplifies on an orbit}\label{sec:proof}

\subsection{The imported comparison}

For a finitely supported probability measure $\mu$ on $\R^I$, let
\[
 B_\eps(\mu)=\sup\bigl\{\E\langle A,\eps\rangle:\ \Law(A)=\mu\bigr\}.
\]
The supremum is over all couplings of $A$ with a vector of independent symmetric signs; $A$ is not required to be independent of $\eps$. Let $\Pi(\mu,\mu)$ denote the probability measures on the product with both marginals equal to $\mu$, and set
\begin{align}
 \RD^\chi_\mu(t)
  &=\inf_{P\in\Pi(\mu,\mu)}
     \left\{\KL(P\Vert\mu\otimes\mu)+\E_P D_t(A,B)\right\},\label{eq:RD}\\
 R^\chi(\mu)&=\int_0^\infty\RD^\chi_\mu(t)\,dt.\nonumber
\end{align}
For finite probability measures, $\KL(P\Vert Q)=\sum_zP(z)\log(P(z)/Q(z))$, with zero summands omitted and value $+\infty$ if $P$ charges a point where $Q$ vanishes. In \eqref{eq:RD}, this relative entropy is the mutual information of the two coordinates.

The Bernoulli consequence of \cite[Eq.~(31)]{invariant} states that
\begin{equation}\label{eq:external}
 B_\eps(\mu)\as R^\chi(\mu)
 \qquad\text{for every finitely supported }\mu,
\end{equation}
with universal constants. We use \eqref{eq:external} as an external theorem. The fixed marginals in \eqref{eq:RD} are essential; they will be checked explicitly rather than discarded in a pointwise minimization.

\subsection{Symmetrizing the maximizing index}

\begin{lemma}\label{lem:uniform}
If $\lambda$ is uniform on $\Omega$, then $B_\eps(\lambda)=S(\Omega)$.
\end{lemma}
\begin{proof}
Every coupling satisfies $\langle A,\eps\rangle\le\max_{a\in\Omega}\langle a,\eps\rangle$, so $B_\eps(\lambda)\le S(\Omega)$.

For the reverse inequality, choose a maximizing index $a(\eps)\in\Omega$, breaking ties by a fixed ordering. Let $U$ be independent and uniform on $\Gamma$, and put
\[
 A=U a(\eps),\qquad \eps'=U\eps.
\]
Conditional on $\eps$, the point $A$ is uniform on $\Omega$, because the action is transitive. The vector $\eps'$ has the same law as $\eps$. Finally, coordinate permutations preserve the pairing, so
\[
 \langle A,\eps'\rangle
 =\langle a(\eps),\eps\rangle
 =\max_{a\in\Omega}\langle a,\eps\rangle.
\]
This is an admissible coupling attaining the desired value.
\end{proof}

\subsection{The Gibbs self-coupling}

\begin{lemma}[Invariant entropy identity]\label{lem:gibbs}
For every $t>0$,
\begin{equation}\label{eq:haar}
 \RD^\chi_\lambda(t)=-\log Z_\Omega(t).
\end{equation}
The minimizing self-coupling is
\begin{equation}\label{eq:gibbs}
 Q_t(a,b)=\frac{\lambda(a)\lambda(b)e^{-D_t(a,b)}}{Z_\Omega(t)}.
\end{equation}
\end{lemma}
\begin{proof}
The row sums of \eqref{eq:gibbs} equal $\lambda(a)$ by the definition of $Z_\Omega(t)$ and its independence of $a$. Symmetry of $D_t$ gives the same column sums. Hence $Q_t\in\Pi(\lambda,\lambda)$.

For any $P\in\Pi(\lambda,\lambda)$, substitution of \eqref{eq:gibbs} yields the exact identity
\begin{equation}\label{eq:KL}
 \KL(P\Vert\lambda\otimes\lambda)+\E_P D_t
 =-\log Z_\Omega(t)+\KL(P\Vert Q_t).
\end{equation}
Relative entropy is nonnegative, and the last term vanishes at $P=Q_t$. This proves both assertions.
\end{proof}

The identity \eqref{eq:KL} is the Haar calculation in \cite[Section~4, Eq.~(23)]{invariant}, written here with all marginal constraints visible. It is precisely transitivity that permits one common normalizing constant.

\begin{proof}[Proof of Theorem~\ref{prop:orbit}]
Combining \eqref{eq:external}, Lemma~\ref{lem:uniform}, and Lemma~\ref{lem:gibbs} gives \eqref{eq:orbit}.

Fix $a_0\in\Omega$ and put $Y_a=\langle a-a_0,\eps\rangle$. Then $Y_{a_0}=0$ and
$\E\max_aY_a=S(\Omega)$. Since the joint law of $(Y_a)$ is symmetric,
\[
 S(\Omega)\le\E\max_a\abs{Y_a}
 \le\E\max_aY_a+\E\max_a(-Y_a)=2S(\Omega).
\]
This proves \eqref{eq:anchor}.

For a Rademacher sum $R_a=\langle a,\eps\rangle$, the elementary moment identities
\[
 \E R_a^2=\norm a_2^2,\qquad
 \E R_a^4\le3\norm a_2^4
\]
and H\"older's inequality give
\begin{equation}\label{eq:khintchine}
 3^{-1/2}\norm a_2\le\E\abs{R_a}\le\norm a_2.
\end{equation}
Indeed, $\norm{R_a}_2\le\norm{R_a}_1^{1/3}\norm{R_a}_4^{2/3}$. The triangle inequality and \eqref{eq:anchor} now give the upper bound in \eqref{eq:absolute}. The lower bound follows from
\[
 S_{\mathrm{abs}}(\Omega)\ge S(\Omega),\qquad
 S_{\mathrm{abs}}(\Omega)\ge\E\abs{\langle a_0,\eps\rangle}.
\]
The maximum of these two lower bounds controls half their sum.
\end{proof}

\begin{remark}[What is, and is not, removed]
After importing \eqref{eq:external}, the stationary application requires no further replication, Bernoulli decomposition, or construction of admissible partitions. This is a simplification of the application, not a claim that the external comparison itself is elementary.
\end{remark}

\section{The equivalent dyadic ball-volume formula}\label{sec:dyadic}

We next recover the scale description in \cite{stationary}. It is useful when the sizes of clipped-distance balls are easier to estimate than the partition function. Fix $r>1$, and for $j\in\Z$ set
\[
 \varphi_j(a,b)=D_{r^{-j}}(a,b)
 =\sum_{i\in I}\min\{r^{2j}\abs{a_i-b_i}^2,1\}.
\]
Let $N_0=1$ and $N_n=2^{2^n}$ for $n\ge1$. For any reference point $a\in\Omega$, define
\begin{equation}\label{eq:jn}
 j_n^\lambda=\sup\bigl\{j\in\Z:\
 \lambda\{b:\varphi_j(a,b)\le2^n\}\ge N_n^{-1}\bigr\},
 \qquad r^{-\infty}=0.
\end{equation}
These values are independent of $a$. The non-strict inequality in the mass condition is part of the convention; in particular, level zero requires the whole orbit.

\begin{theorem}[Dyadic form]\label{thm:dyadic}
With constants depending only on $r$,
\begin{equation}\label{eq:dyadic}
 \cH(\Omega)\as_r\sum_{n\ge0}2^n r^{-j_n^\lambda}.
\end{equation}
Consequently the sum in \eqref{eq:dyadic} estimates the signed supremum and the anchored absolute supremum. The unanchored absolute supremum has the additional term $\norm{a_0}_2$.
\end{theorem}

For $a_g=(f(gx))_x$, uniform measure on $G$ pushes forward to $\lambda$. Thus \eqref{eq:jn} is exactly the definition using normalized counting measure $\mu$ and
$\varphi_j(g,h)=\sum_x\min\{r^{2j}\abs{f(gx)-f(hx)}^2,1\}$.

\subsection{From a Laplace transform to threshold radii}

Put $\Delta=\diam(\Omega,\norm\cdot_2)$ and $F(t)=-\log Z_\Omega(t)$. For $n\ge1$, let
\begin{equation}\label{eq:tau}
 \tau_n=\inf\bigl\{t>0:\
 \lambda\{b:D_t(a,b)\le2^n\}\ge2^{-2^n}\bigr\}.
\end{equation}
The sequence $(\tau_n)$ is nonincreasing, $\tau_n\le\Delta$, and it is eventually zero. For the last assertion, the point $a$ alone has sufficient mass once $2^{-2^n}\le\abs{\Omega}^{-1}$.

\begin{lemma}\label{lem:laplace}
For every non-singleton orbit,
\begin{equation}\label{eq:laplace}
 \int_0^\infty F(t)\,dt\as\Delta+\sum_{n\ge1}2^n\tau_n.
\end{equation}
\end{lemma}
\begin{proof}
We first locate the height of $F$ at the threshold $\tau_n$ and then account for the large-scale tail.

\step{Threshold estimates}
If $t<\tau_n$, split the average defining $Z_\Omega(t)$ according to whether $D_t(a,b)\le2^n$. This gives
\[
 Z_\Omega(t)\le2^{-2^n}+e^{-2^n}\le2\,2^{-2^n},
 \qquad
 F(t)\ge(2^n-1)\log2\ge c2^n.
\]
If $t>\tau_n$, monotonicity of the balls gives
\[
 Z_\Omega(t)\ge2^{-2^n}e^{-2^n},
 \qquad F(t)\le(1+\log2)2^n.
\]
For each $t$, the geometric sum of $2^n$ over levels with $t<\tau_n$ is bounded by a constant times its largest term, and therefore by $CF(t)$. Integrating yields
\begin{equation}\label{eq:tau-lower}
 \sum_{n\ge1}2^n\tau_n\lesssim\int_0^\infty F(t)\,dt.
\end{equation}
Conversely, on $(\tau_{n+1},\tau_n)$ the upper estimate at level $n+1$ applies. The intervals with positive length partition $(0,\tau_1)$, up to endpoints, because the sequence is eventually zero. Hence
\[
 \int_0^{\tau_1}F(t)\,dt
 \lesssim\sum_{n\ge1}2^{n+1}(\tau_n-\tau_{n+1})
 \lesssim\sum_{n\ge1}2^n\tau_n.
\]
On $(\tau_1,\Delta)$, the level-one estimate bounds $F$ by a universal constant.

\step{The diameter term}
If $t\ge\Delta$, no clipping occurs, and $D_t(a,b)=\norm{a-b}_2^2/t^2\le1$. The inequalities
$1-e^{-u}\ge(1-e^{-1})u$ for $0\le u\le1$ and $-\log\E e^{-U}\le\E U$ imply
\[
 F(t)\as\frac{\E_{b\in\Omega}\norm{a-b}_2^2}{t^2}.
\]
The numerator is independent of $a$. Choose $u,v\in\Omega$ with $\norm{u-v}_2=\Delta$ and average
$\norm{u-v}_2^2\le2\norm{u-b}_2^2+2\norm{b-v}_2^2$. This gives
\[
 \frac{\Delta^2}{4}\le\E_{b\in\Omega}\norm{a-b}_2^2\le\Delta^2.
\]
It follows that $\int_\Delta^\infty F(t)\,dt\as\Delta$. Combining this with the threshold bounds proves \eqref{eq:laplace}.
\end{proof}

\subsection{The geometric grid and level zero}

\begin{proof}[Proof of Theorem~\ref{thm:dyadic}]
For $\tau_n>0$, the admissible set in \eqref{eq:tau} is $[\tau_n,\infty)$. This follows from continuity of each $D_t(a,b)$ and finiteness of the orbit. Comparing with the grid $t=r^{-j}$ gives
\begin{equation}\label{eq:grid}
 \tau_n\le r^{-j_n^\lambda}<r\tau_n\qquad(n\ge1).
\end{equation}
When $\tau_n=0$, all positive scales are admissible and both sides of the comparison vanish.

It remains to identify level zero. Vectors in a permutation orbit have the same coordinate sum, so two distinct vectors cannot differ in only one coordinate. If $D_t(a,b)\le1$ and $a\ne b$, no coordinate contribution can equal~$1$: otherwise all other differences would have to vanish. Thus no clipping occurs and $\norm{a-b}_2\le t$. The converse is immediate. Transitivity makes the largest distance from a fixed point equal to $\Delta$, and consequently
\begin{equation}\label{eq:level-zero}
 \max_{b\in\Omega}D_t(a,b)\le1
 \quad\Longleftrightarrow\quad t\ge\Delta.
\end{equation}
For $\Delta>0$, this gives $\Delta\le r^{-j_0^\lambda}<r\Delta$. Now combine \eqref{eq:grid} with Lemma~\ref{lem:laplace}. For a singleton orbit, every $j_n^\lambda=\infty$ and both sides of \eqref{eq:dyadic} are zero.
\end{proof}

\section{Overlap profiles of set orbits}\label{sec:overlap}

Let $\cA$ be the orbit of a set $A\subset I$ under $\Gamma$. All members have cardinality $m=\abs A$. For $B\in\cA$, define
\begin{equation}\label{eq:overlap-profile}
 \kappa(A,B)=\abs{A\mathbin\triangle B}=2(m-\abs{A\cap B}),\qquad
 \cL_A(s)=-\log\!\left(\frac1{\abs{\cA}}\sum_{B\in\cA}e^{-s\kappa(A,B)}\right).
\end{equation}
The reference set $A$ is immaterial. We call $\cL_A$ the overlap profile. It is a Laplace-transform description of the distribution of symmetric-difference sizes, not an independence assumption about the random sums.

\begin{corollary}[Overlap formula]\label{cor:overlap}
For $\one_{\cA}=\{\one_B:B\in\cA\}$, put
\[
 H_A=\cL_A(1)+\frac12\int_0^1\cL_A(s)s^{-3/2}\,ds.
\]
Then, with universal constants,
\begin{align}
 \E\max_{B\in\cA}\sum_{i\in B}\eps_i&\as H_A,\label{eq:overlap-signed}\\
 \E\max_{B\in\cA}\left|\sum_{i\in B}\eps_i\right|
   &\as\sqrt m+H_A.\label{eq:overlap}
\end{align}
\end{corollary}
\begin{proof}
Every nonzero coordinate difference between two indicator vectors has absolute value~$1$. Therefore
\[
 D_t(\one_A,\one_B)=\kappa(A,B)\chi(1/t),\qquad
 -\log Z_{\one_{\cA}}(t)=\cL_A(\chi(1/t)).
\]
On $0<t\le1$, the last expression is $\cL_A(1)$. On $t\ge1$, substitute $s=t^{-2}$ to obtain the displayed formula for $H_A=\cH(\one_{\cA})$. Apply Theorem~\ref{prop:orbit}, noting that $\norm{\one_A}_2=\sqrt m$.
\end{proof}

For translates in a finite group, the average in \eqref{eq:overlap-profile} can equivalently be taken over group elements, including multiplicities. Thus the distribution of overlaps determines the order of the expected maximum even when the random sums are strongly dependent.

\subsection{A scalar calculation}

\begin{lemma}\label{lem:scalar}
For $M\ge2$ and $v\ge0$, set
$q_M(v)=-\log[M^{-1}+(1-M^{-1})e^{-v}]$. Then
\begin{equation}\label{eq:scalar}
 \tfrac13\min\{v,\log M\}\le q_M(v)\le\min\{v,\log M\}.
\end{equation}
Moreover, for $a,L>0$,
\begin{equation}\label{eq:scalar-integral}
 \int_0^\infty\min\{a\chi(1/t),L\}\,dt
 =2\min\{a,\sqrt{aL}\}.
\end{equation}
\end{lemma}
\begin{proof}
The upper bounds in \eqref{eq:scalar} follow by retaining, respectively, $e^{-v}$ and $M^{-1}$ as lower bounds on the expression inside the logarithm. Also
\[
 q_M'(v)=\frac{(M-1)e^{-v}}{1+(M-1)e^{-v}}\ge\frac13
 \qquad(0\le v\le\log M).
\]
Integrate up to $\min\{v,\log M\}$ and use monotonicity thereafter. For \eqref{eq:scalar-integral}, if $L\ge a$ the integral is $a\int_0^\infty\chi(1/t)\,dt=2a$. If $L<a$, split at $t=\sqrt{a/L}\ge1$ to obtain $2\sqrt{aL}$.
\end{proof}

\subsection{Separated orbits and independent blocks}

\begin{corollary}[Uniformly separated set orbits]\label{cor:separated}
Suppose $m\ge1$, $M=\abs{\cA}\ge2$, and $0<\delta\le1$. If
\begin{equation}\label{eq:sep}
 \abs{B\cap C}\le(1-\delta)m
 \qquad(B,C\in\cA,\ B\ne C),
\end{equation}
then
\begin{equation}\label{eq:sep-bound}
 \E\max_{B\in\cA}\left|\sum_{i\in B}\eps_i\right|
 \as_\delta\min\{m,\sqrt{m\log M}\}.
\end{equation}
The same estimate holds without absolute values.
\end{corollary}
\begin{proof}
For distinct sets, $2\delta m\le\kappa(B,C)\le2m$. Hence
\[
 M^{-1}+(1-M^{-1})e^{-2ms}
 \le e^{-\cL_A(s)}
 \le M^{-1}+(1-M^{-1})e^{-2\delta ms}.
\]
Lemma~\ref{lem:scalar} gives $\cL_A(s)\as_\delta\min\{ms,\log M\}$. After substitution of $s=\chi(1/t)$, \eqref{eq:scalar-integral} proves the signed estimate. The additional $\sqrt m$ in \eqref{eq:overlap} is absorbed: $m\ge1$ and $\log M\ge\log2$ imply
$\min\{m,\sqrt{m\log M}\}\ge c\sqrt m$.
\end{proof}

If $A$ is a subgroup of size $m$ and index $M$, its distinct translates are disjoint. Their partition function is exactly
$M^{-1}+(1-M^{-1})e^{-2ms}$. More generally, for $M$ independent sums $R_1,\ldots,R_M$, each of $m$ independent signs,
\begin{equation}\label{eq:blocks}
 \E\max_{1\le j\le M}\abs{R_j}
 \as\min\{m,\sqrt{m\log(eM)}\},\qquad m,M\ge1.
\end{equation}
For $M\ge2$, view the blocks as an orbit under permutations of whole blocks and apply Corollary~\ref{cor:separated}. For $M=1$, use \eqref{eq:khintchine}. Without absolute values, the same order holds for $M\ge2$, but the expectation is zero for $M=1$.

\begin{example}[Why clipping is necessary]\label{ex:singleton}
Take $m=1$ and $M=\abs I\ge2$. Then
\[
 \E\max_{i\in I}\abs{\eps_i}=1,\qquad
 \E\max_{i\in I}\eps_i=1-2^{1-M}.
\]
In contrast, a purely quadratic cost would give an integral of order $\sqrt{\log M}$. The cap in $\chi$ records the saturation of a sign at size~$1$.
\end{example}

\subsection{A two-level overlap profile}

Suppose $M\ge2$ and every two distinct members have the same intersection size $\ell<m$. Write $v=m-\ell$. Then
\[
 e^{-\cL_A(s)}=M^{-1}+(1-M^{-1})e^{-2vs}.
\]
The preceding calculation, without assuming $v\as m$, yields
\begin{align}
 \E\max_{B\in\cA}\sum_{i\in B}\eps_i
   &\as\min\{v,\sqrt{v\log M}\},\label{eq:two-level-signed}\\
 \E\max_{B\in\cA}\left|\sum_{i\in B}\eps_i\right|
   &\as\sqrt m+\min\{v,\sqrt{v\log M}\}.\label{eq:two-level}
\end{align}
This includes any difference-set example with the stated constant-overlap property. The distinction between $m$ and $v$ is essential: a large common overlap can leave the size of each random sum almost unchanged while making the variation across the family much smaller.

\section{Moving intervals and rectangular windows}\label{sec:windows}

\subsection{An integration estimate}

The upper bounds for geometric windows will use one elementary calculation.
\begin{lemma}\label{lem:integration}
For $m>0$, $L\ge1$, and $d\ge1$,
\begin{equation}\label{eq:integration}
 \int_0^\infty\min\!\left\{\frac m{t^2},\,
 L+d\log\!\left(1+\frac m{t^2}\right)\right\}\,dt
 \le C_d\sqrt{mL}.
\end{equation}
\end{lemma}
\begin{proof}
Let $a=\sqrt{m/L}$. On $(a,\infty)$, integrate $m/t^2$ to get $m/a=\sqrt{mL}$. On $(0,a)$, write $t=au$. For $0<u\le1$,
\[
 \log(1+L/u^2)\le\log(1+L)+2\log(1/u).
\]
Since $\int_0^1\log(1/u)\,du=1$, the contribution is at most
$a[L+d\log(1+L)+2d]\le C_d aL$. This proves \eqref{eq:integration}.
\end{proof}

\subsection{Moving intervals on a cycle}

Let $N\ge1$, $1\le m\le N$, and define, with indices understood modulo $N$,
\[
 W_{k,m}=\sum_{j=0}^{m-1}\eps_{k+j},\qquad 0\le k<N.
\]
Each window has exactly $m$ coordinates.

\begin{theorem}[Cyclic moving windows]\label{thm:interval}
Uniformly in $N$ and $m$,
\begin{equation}\label{eq:interval}
 \E\max_{0\le k<N}\abs{W_{k,m}}
 \as\min\!\left\{m,\sqrt{m\log\frac{eN}{m}}\right\}.
\end{equation}
For the signed maximum, put $u=\min\{m,N-m\}$. If $u\ge1$, then
\begin{equation}\label{eq:interval-signed}
 \E\max_{0\le k<N}W_{k,m}
 \as\min\!\left\{u,\sqrt{u\log\frac{eN}{u}}\right\}.
\end{equation}
If $m=N$, the signed expectation is zero.
\end{theorem}

\begin{proof}
We first treat $m\le N/2$ by overlap geometry, and then pass to larger windows by complementation.

\step{The exact overlap profile}
For $A=\{0,\ldots,m-1\}\subset\Z/N\Z$,
\[
 \kappa(k)=\abs{A\mathbin\triangle(k+A)}
 =2\min\{k,N-k,m\},\qquad 0\le k<N.
\]
Consequently
\begin{equation}\label{eq:interval-partition}
 e^{-\cL_A(s)}=\frac1N\left[
 1+2\sum_{k=1}^{m-1}e^{-2ks}+(N-2m+1)e^{-2ms}\right].
\end{equation}
The coefficient of $e^{-2ms}$ counts the shifts with disjoint windows.

\step{The upper bound}
Since $\kappa(k)\le2m$, we have $e^{-\cL_A(s)}\ge e^{-2ms}$. Also, for $s>0$,
\[
 \sum_{k=0}^{m-1}e^{-2ks}
 \ge\int_0^m e^{-2sx}\,dx
 =\frac{1-e^{-2ms}}{2s}
 \ge\frac{cm}{1+ms}.
\]
The same bound holds at $s=0$ by continuity. Retaining these terms in \eqref{eq:interval-partition} gives
\begin{equation}\label{eq:interval-lower-Z}
 e^{-\cL_A(s)}\ge\frac{cm}{N(1+ms)}.
\end{equation}
\Needspace{7\baselineskip}
Let $L=\log(eN/m)$. The two lower bounds on the partition function give
\begin{align*}
 \cL_A(\chi(1/t))&\le2m\chi(1/t),\\*
 \cL_A(\chi(1/t))&\le
 C\min\!\left\{\frac m{t^2},\,
 L+\log\!\left(1+\frac m{t^2}\right)\right\}.
\end{align*}
Integrating the first estimate gives $H_A\le4m$; integrating the second and applying Lemma~\ref{lem:integration} gives $H_A\lesssim\sqrt{mL}$. The baseline $\sqrt m$ is absorbed by $\min\{m,\sqrt{mL}\}$. Corollary~\ref{cor:overlap} therefore proves the upper bound in \eqref{eq:interval}.

\step{The lower bound}
The windows starting at $0,m,2m,\ldots,(M_0-1)m$, where $M_0=\lfloor N/m\rfloor\ge2$, are disjoint. Their sums are independent, so \eqref{eq:blocks} gives the required lower bound. Indeed,
$\log M_0\as\log(eN/m)$ for $N/m\ge2$. The signed independent-block estimate gives the same lower bound without absolute values. This proves both conclusions when $m\le N/2$.

\step{Large windows and the endpoint}
Let $m>N/2$ and $u=N-m$. Write $T_N=\sum_{x=0}^{N-1}\eps_x$. Then
\begin{equation}\label{eq:complement}
 W_{k,m}=T_N-W_{k+m,u},
\end{equation}
where the empty-window sum is defined to be zero. A single length-$m$ sum gives the absolute lower bound $c\sqrt m\ge c'\sqrt N$. The upper bound follows from \eqref{eq:complement}, $\E\abs{T_N}\le\sqrt N$, and the already proved estimate for $u\le N/2$, because
\[
 u\log(eN/u)\le N\qquad(0<u\le N).
\]
Thus the absolute maximum is comparable to $\sqrt N$. For $m>N/2$, the right-hand side of \eqref{eq:interval} is also comparable to $\sqrt N$.

For the signed maximum, take the maximum in \eqref{eq:complement} and then expectations. Since $\E T_N=0$ and the vector of complementary-window sums is symmetric,
\begin{equation}\label{eq:signed-complement}
 \E\max_k W_{k,m}=\E\max_k W_{k,u}.
\end{equation}
No independence between $T_N$ and the complementary windows is used. This proves \eqref{eq:interval-signed}. At $m=N$, all windows equal $T_N$, so the signed expectation is zero.
\end{proof}

For the range treated in \cite[Theorem~2.1]{stationary}, this is the overlap-profile form of the moving-interval estimate, with the convention of exactly $m$ summands. The signed and absolute extensions cannot be identified for all $m$. At $m=N$, their orders are respectively $0$ and $\sqrt N$; at $m=N-1$, the signed expectation is of constant order while the absolute maximum remains of order $\sqrt N$.

\subsection{Rectangular scan windows}

Let
\[
 G=\prod_{j=1}^d\Z/N_j\Z,\qquad
 A=\prod_{j=1}^d\{0,\ldots,m_j-1\},
 \qquad 1\le m_j\le N_j/2.
\]
Write $m=\prod_jm_j$ and $Q=\prod_j(N_j/m_j)$.

\begin{theorem}\label{thm:boxes}
With constants depending only on $d$,
\begin{equation}\label{eq:boxes}
 \E\max_{g\in G}\left|\sum_{x\in g+A}\eps_x\right|
 \as_d\min\{m,\sqrt{m\log(eQ)}\}.
\end{equation}
The same estimate holds for the signed maximum.
\end{theorem}
\begin{proof}
For shifts satisfying $0\le k_j<m_j$, the overlap factorizes:
\[
 \abs{A\cap(k+A)}=m\prod_{j=1}^d(1-k_j/m_j).
\]
The inequality $1-\prod_j(1-u_j)\le\sum_ju_j$ for $u_j\in[0,1]$ gives
$\kappa(k)\le2m\sum_j k_j/m_j$. Retaining only these shifts and comparing each geometric sum with an integral yields
\begin{align*}
 e^{-\cL_A(s)}
 &\ge\prod_{j=1}^d\left[\frac1{N_j}
       \sum_{k=0}^{m_j-1}e^{-2msk/m_j}\right]\\
 &\ge\frac{c^d}{Q(1+ms)^d}.
\end{align*}
Also $e^{-\cL_A(s)}\ge e^{-2ms}$. Hence $H_A\lesssim m$, while Lemma~\ref{lem:integration}, with $L=\log(eQ)$ and constants allowed to depend on $d$, gives $H_A\lesssim_d\sqrt{mL}$. The baseline $\sqrt m$ is absorbed as before.

A rectangular grid contains $M_0=\prod_j\lfloor N_j/m_j\rfloor$ disjoint translates. For $x\ge2$, $\sqrt x\le\lfloor x\rfloor\le x$, so
$\sqrt Q\le M_0\le Q$ and $\log M_0\as\log(eQ)$. Applying the signed and absolute independent-block estimates proves both lower bounds.
\end{proof}

\section{Affine subspaces over a finite field}\label{sec:finite}

The interval and box examples use many nearby, strongly overlapping sets. Finite geometry gives a complementary example in which every two distinct sets are uniformly separated.

Let $q$ be a prime power, let $1\le k<d$, and let $\cA$ be the family of all affine $k$-dimensional subspaces of $\Fq^d$. The affine group acts by coordinate permutations on $I=\Fq^d$ and transitively on $\cA$. Every member has $m=q^k$ points. Two distinct members are either disjoint or meet in an affine subspace of dimension at most $k-1$, so
\[
 \abs{B\cap C}\le q^{k-1}\le\tfrac12q^k
 \qquad(B\ne C).
\]
Thus Corollary~\ref{cor:separated} applies with a universal separation constant.

\begin{theorem}\label{thm:affine}
For independent symmetric signs indexed by $\Fq^d$,
\begin{equation}\label{eq:affine}
 \E\max_{V\in\cA}\left|\sum_{x\in V}\eps_x\right|
 \as\min\!\left\{q^k,\,
 \sqrt{q^k(k+1)(d-k)\log q}\right\}.
\end{equation}
The constants are universal, and the same estimate holds without absolute values.
\end{theorem}
\begin{proof}
It remains to count the family. Counting ordered independent $k$-tuples in $\Fq^d$ and dividing by the number of ordered bases of a fixed $k$-dimensional space gives
\[
 \genfrac{[}{]}{0pt}{}{d}{k}_q
 =\prod_{j=0}^{k-1}\frac{q^{d-j}-1}{q^{k-j}-1}.
\]
Each linear space has $q^{d-k}$ distinct affine translates. Therefore
\begin{equation}\label{eq:affine-count}
 M=\abs{\cA}=q^{d-k}\genfrac{[}{]}{0pt}{}{d}{k}_q,
 \qquad
 q^{(k+1)(d-k)}\le M\le4q^{(k+1)(d-k)}.
\end{equation}
For completeness, after extracting $q^{k(d-k)}$, the remaining product has factors
$(1-q^{-(d-j)})/(1-q^{-(k-j)})$. Each is at least~$1$, and the whole product is at most
$\prod_{i=1}^\infty(1-2^{-i})^{-1}<4$. One elementary check of the last inequality is
\[
 \prod_{i=1}^\infty(1-2^{-i})
 \ge\frac{21}{64}\left(1-\sum_{i=4}^\infty2^{-i}\right)
 =\frac{147}{512}>\frac14.
\]
Taking logarithms in \eqref{eq:affine-count} gives
$\log M\as(k+1)(d-k)\log q$. Substitute this and $m=q^k$ into Corollary~\ref{cor:separated}.
\end{proof}

The estimate does not require the subspace sums to be independent. Their dependence is substantial, but the uniform overlap bound places the full family within the separated-orbit calculation. The restriction $k<d$ also has a transparent meaning: for $k=d$ there is only one set, so the signed expectation is zero and the absolute expectation is comparable to $q^{d/2}$.

\section{The same overlaps with different coordinate laws}\label{sec:tails}

The overlap profile describes the sets; the coordinate cost describes the random variables. This separation is useful beyond signs. We state the normalization and the additional imported result explicitly.

Let $(Y_i)_{i\in I}$ be independent copies of a nonzero symmetric random variable $Y$. Suppose
\[
 N(u)=-\log\Pp(\abs Y\ge u),\qquad u\ge0,
\]
is closed, convex, and nondecreasing, with $N(0)=0$ and the value $+\infty$ allowed. This is the log-concave-tail assumption. Write
$q(p)=\sup\{u\ge0:N(u)\le p\}$ and first normalize $q(1)=1$. Let $b=\lim_{p\to\infty}q(p)$. Following \cite[Section~1]{invariant}, define the even potential
\[
 V(u)=
 \begin{cases}
 u^2,&\abs u\le1,\\
 2N(\abs u)-1,&1<\abs u\le b,\\
 +\infty,&\abs u>b,
 \end{cases}
\]
when $b>1$. If $b=1$, take $V(u)=u^2$ on $[-1,1]$ and $+\infty$ outside. Its intrinsic coordinate cost is
\begin{equation}\label{eq:intrinsic}
 \Psi(r)=\int_0^b V'_+(s)\one_{\{V'_+(s)<r\}}\,ds,
 \qquad r\ge0.
\end{equation}
The strict inequality fixes the convention at slope thresholds.

Theorem~1.1 of \cite{invariant} is the analogue of \eqref{eq:external} for this source, with $D_t^\Psi(a,b)=\sum_i\Psi(\abs{a_i-b_i}/t)$. We import this general-tail comparison only in the present section. Define $\cL_A(\infty)=\log M$, where $M=\abs{\cA}$; this is the limit of \eqref{eq:overlap-profile}, since only the reference set has zero symmetric difference.

\begin{corollary}\label{cor:tails}
Under the preceding assumptions,
\begin{equation}\label{eq:tails}
 \E\max_{B\in\cA}\sum_{i\in B}Y_i
 \as\int_0^\infty\cL_A\bigl(\Psi(1/t)\bigr)\,dt.
\end{equation}
For the absolute maximum, the right-hand side must be augmented by
$\E\abs{\sum_{i\in A}Y_i}$. If the law is not normalized, apply the formula to $Y/q(1)$ and multiply the resulting estimate by $q(1)$.
\end{corollary}
\begin{proof}
The law of the coordinate vector is invariant under permutations, so Lemma~\ref{lem:uniform} applies unchanged. The Gibbs computation also applies to the cost $D_t^\Psi$. When the cost can be infinite, use $e^{-\infty}=0$ and restrict the entropy identity to finite-cost couplings; the diagonal ensures a positive normalizing constant and a feasible coupling. For indicator vectors,
\[
 D_t^\Psi(\one_A,\one_B)=\kappa(A,B)\Psi(1/t),
\]
with $0\cdot\infty=0$. The general-tail comparison therefore gives \eqref{eq:tails}. Symmetry and the same triangle-inequality argument as in Theorem~\ref{prop:orbit} give the absolute version, retaining the single-index expectation rather than replacing it without justification.
\end{proof}

\subsection{Three coordinate laws on a separated orbit}

Suppose \eqref{eq:sep} holds and $M\ge2$. By Lemma~\ref{lem:scalar}, including its limiting value at infinity, the signed supremum is comparable to
\begin{equation}\label{eq:tails-separated}
 \int_0^\infty\min\{m\Psi(1/t),L\}\,dt,
 \qquad L=\log M.
\end{equation}
This gives the following estimates. In the last row, $Y=\eps E$, where $E$ is an independent mean-one exponential, so $Y$ has density $e^{-|y|}/2$.

\begin{center}
\renewcommand{\arraystretch}{1.45}
\setlength{\tabcolsep}{9pt}
\begin{tabular}{@{}p{0.48\textwidth}p{0.40\textwidth}@{}}
\toprule
\textbf{Independent coordinate law} & \textbf{Expected maximum}\\
\midrule
Symmetric signs & $\min\{m,\sqrt{mL}\}$\\
Standard Gaussian & $\sqrt{mL}$\\
Symmetric mean-one exponential & $\sqrt{mL}+L$\\
\bottomrule
\end{tabular}
\end{center}
\noindent The table applies to both signed and absolute maxima, with constants depending only on $\delta$.

For signs, \eqref{eq:intrinsic} gives $\Psi(r)=\chi(r/2)$; rescaling $t$ reduces the calculation to \eqref{eq:scalar-integral}. For a normalized Gaussian law, the Gaussian-tail estimates give $V'_+(s)\as s$, and hence $\Psi(r)\as r^2$ by \eqref{eq:intrinsic}. Undoing the fixed normalization changes only universal constants. Thus the relevant integral is
\[
 \int_0^\infty\min\{m/t^2,L\}\,dt=2\sqrt{mL}.
\]
For the symmetric mean-one exponential,
\[
 \Psi(r)=
 \begin{cases}
 r^2/4,&0\le r\le2,\\
 +\infty,&r>2.
 \end{cases}
\]
Indeed, $V'_+(s)=2s$ on $0<s<1$ and $V'_+(s)=2$ on $s\ge1$. Consequently \eqref{eq:tails-separated} equals
\begin{equation}\label{eq:exponential-integral}
 \frac L2+\int_{1/2}^\infty\min\{m/(4t^2),L\}\,dt
 =\begin{cases}
 \sqrt{mL},&L\le m,\\
 (m+L)/2,&L\ge m.
 \end{cases}
\end{equation}
Both cases are comparable to $\sqrt{mL}+L$.

\Needspace{6\baselineskip}
Finally, the single-index absolute expectation is comparable to $\sqrt m$ for each of the three laws. For signs this is \eqref{eq:khintchine}; for Gaussians it follows from the one-dimensional normal distribution. For symmetric mean-one exponentials, the sum $R$ of $m$ coordinates satisfies
$\E R^2=2m$ and $\E R^4=12m^2+12m\le24m^2$, so the same H\"older argument applies. Since $L\ge\log2$, these baseline terms are absorbed by the entries in the table.

The three rows isolate the role of the coordinate law. The same overlap geometry produces Bernoulli saturation at $m$, Gaussian growth at scale $\sqrt{mL}$, or an additional exponential contribution of size~$L$. The sets have not changed; only the cost inserted into their overlap profile has changed.

\end{document}